\documentclass[12pt,a4paper]{article}
\usepackage[centertags]{amsmath}
\usepackage{amsfonts}

\usepackage{graphics}
\usepackage{amssymb}
\usepackage{amsthm}
\usepackage{newlfont}
\usepackage{epsfig}
\usepackage{amssymb}
\usepackage{amsmath}
\usepackage{amsthm}
\usepackage{graphicx}
\usepackage{makeidx}
 \usepackage{lineno}
\theoremstyle{plain}
\newtheorem{thm}{Theorem}[section]

\newtheorem{prop}[thm]{Proposition}
\theoremstyle{definition}

\newtheorem{rem}{Remark}[section]
\theoremstyle{remark} \tolerance=10000 \hbadness=10000
\def \ni{\noindent}

\author{
	Anu V.\footnote{E-mail : anusaji1980@gmail.com}\\ Department of Mathematics\\
	St.Peter's College\\ Kolenchery - 682 311\\  Kerala, India.\and
	Aparna Lakshmanan S.\footnote{E-mail : aparnaren@gmail.com}\\
	Department of Mathematics\\St.Xavier's College for Women\\Aluva -
	683 101\\\vspace{0.2cm} Kerala, India.}

\title{On the Double Roman Domination Number of Generalized Sierpi\'{n}ski Graphs}

\date{}
\begin{document}

\maketitle

\begin{abstract}

 In
this paper, we  study the double Roman domination number of generalized Sierpi\'{n}ski graphs $S(G,t)$. More precisely, we obtain a  bound for the  double Roman domination number of $S(G, t)$. We also find the exact value of $\gamma_{dR}(S(K_{n}, 2))$.\\


\ni {\bf Keywords:} Double Roman Dominating Function, Double Roman Domination Number, Sierpi\'{n}ski Graphs.\\

 \ni{\bf AMS
Subject Classification:} 05C69; 05C76


\end{abstract}


\section{Introduction}
Let $G = (V(G),E(G))$ be a graph with vertex set $V(G)$ and edge
set $E(G)$. If there is no ambiguity in the choice of $G$, then we
write $V(G)$ and $E(G)$ as $V$ and $E$, respectively.  Let
$f:V\rightarrow \{0,1,2,3\}$ be a function defined on $V(G). $  Let $V_{i}^{f}=\{v\in V(G) : f(v)=i\}.$ (If there is no ambiguity, $V_{i}^{f}$ is written as $V_{i}$.) Then $f$ is a double Roman dominating function  
 (DRDF)
on  $G$ if it satisfies the  following conditions.\\
(i) If $v\in V_{0}$, then vertex $v$ must have at least two
neighbors
in $V_{2}$ or at least one neighbor in $V_{3}$.\\
(ii) If $v\in V_{1}$, then vertex $v$ must have at least one
neighbor
in $V_{2} \cup V_{3}$.\\

The weight of a DRDF $f$ is
the sum $f(V)=\sum_{v\in V}f(v)$. The double Roman domination
number, $\gamma_{dR}(G)$, is the minimum among the weights of  DRDFs on $G$, and a DRDF
on $G$ with weight $\gamma_{dR}(G)$ is called a
$\gamma_{dR}$-function of $G$
\cite{Bee}.\\

Let $(V_{0}, V_{1}, V_{2}, V_{3})$ be the ordered partition of $V$ induced by $f$. Note that there exists a $1-1$ correspondence between the functions $f$ and the ordered partitions $(V_{0}, V_{1}, V_{2}, V_{3})$ of $V$. Thus we will write $f=(V_{0}, V_{1}, V_{2}, V_{3})$.\\

 R. A. Beeler, T. W.
Haynes and S. T. Hedetniemi pioneered  the study of double Roman domination in \cite{Bee}. The relationship between double Roman domination and Roman
domination and  the bounds on the double Roman domination number of
a graph $G$ in terms of its domination number were discussed by them. They also determined a sharp upper bound on $\gamma_{dR}(G)$ in terms of the
order of $G$ and characterized the graphs attaining this bound. In \cite{Abd1}, it was verified that the decision problem associated with  $\gamma_{dR}(G)$ is NP-complete for bipartite and chordal graphs. Above all this, a characterization of graphs $G$ with small $\gamma_{dR}(G)$ was provided. In \cite{Hao}, G. Hao et al. introduced the study of the double Roman domination of digraphs and  L. Volkmann proposed a sharp lower bound on $\gamma_{dR}(G)$ in \cite{Vol}. In \cite{Anu}, it was proved that  $\gamma_{dR}(G)+2 \leqslant \gamma_{dR}(M(G)) \leqslant \gamma_{dR}(G)+3$,
where $M(G)$ is the Mycielskian graph of $G$ and a construction was also given which  confirms that there
is no relation between the double Roman domination number of a
graph and its induced subgraphs. The impact of  some graph operations such as corona, cartesian product and addition of twins,  on 
double Roman domination number was studied in \cite{Anu1}. In \cite{Amj}, J. Amjadi et al. improved the upper bound on  $\gamma_{dR}(G)$ given in \cite{Bee} by showing that for any connected graph $G$ of order $n$ with minimum degree at least two, $\gamma_{dR}(G)\leqslant \frac{8n}{7}$.

\subsection{Basic Definitions and Preliminaries}
The open
neighborhood of  a vertex $v\in V$ is the set $N(v)=\{u: uv \in
E\}$, and its closed neighborhood is $N[v]=N(v)\cup \{v\}$.
The vertices in $ N(v)$ are called the neighbors of $v$.  When $G$ must be explicit, these open and closed neighborhoods are denoted by $N_{G}(v)$ and $N_{G}[v]$, respectively.   $|N(v)|$ is called the degree of the vertex $v$ in $G$ and is denoted by $d_{G}(v)$, or simply $d(v)$. A vertex of degree $0$ is known as an isolated vertex of $G$. \par
 If $U$ is a non-empty subset of the vertex set $V$ of the graph $G$ then the subgraph $<U>$ of $G$ induced by $U$ is defined as the graph having vertex set $U$ and edge set consisting of those edges of $G$ that have both ends in $U$. A subset $S$ of the vertex set $V$ of a graph $G$ is called independent if no two vertices of $S$ are adjacent in $G$. $S\subseteq V$ is a maximum independent set of $G$ if $G$ has no independent set $S'$ with $|S'|>|S|$. The number of vertices in a maximum independent set of $G$ is called the independence number,  denoted by $\alpha(G)$. A complete graph on $n$ vertices, denoted by  $K_{n}$, is the graph in which any two vertices are adjacent.\\

 A Roman dominating function (RDF) on a graph $G=(V,E)$ is defined as a function $f:V\rightarrow \{0,1,2\}$ satisfying the condition that every vertex $v$ for which $f(v)=0$ is adjacent to at least one vertex $u$ for which $f(u)=2$. The weight of a RDF is the value $f(V)=\sum_{v\in V}f(v)$. The Roman domination number of a graph $G$, denoted by $\gamma_{R}(G)$, is the minimum among the weights of  RDFs on $G$.\\

 Let $G=(V,E)$ be a non-empty graph of order $n\geqslant 2$, and $t$ a positive integer. We denote by $V^{t}$ the set of words of length $t$ on alphabet $V$. The letters of a word $u$ of length $t$ are denoted by $u_{1}u_{2}\ldots u_{t}$. Klav\v{z}ar and Milutinovi\'{c} introduced in \cite{Kla} the graph $S(K_{n},t), t\geqslant 1$, ($S(t,n)$ in their notation) whose vertex set is $V^{t},$ where $\{u,v\}$ is an edge if and only if there exists $i\in \{1, 2,\ldots,t\}$ such that:
 \begin{center}
 	(i) $u_{j}=v_{j}, $ if $j< i$; (ii) $u_{i}\neq v_{i}$; (iii) $u_{j}=v_{i}$ and $v_{j}=u_{i}$ if $j>i$.
 \end{center}
 Later, those graphs have been called Sierpi\'{n}ski graphs in \cite{Kla1}. This construction was generalized in \cite{Gra} for
 any graph $G=(V,E)$, by defining the  $t^{th}$ generalized Sierpi\'{n}ski graph of $G$, denoted by $S(G,t)$, as the graph with  vertex set $V^{t}$ and  edge set $\{\{wu_{i}u_{j}^{r-1}, wu_{j}u_{i}^{r-1}\}: \{u_{i}, u_{j}\}\in E, i\neq j; r\in \{1, 2,\ldots,t\}; w\in V^{t-r}\}$. (See Figure \ref{Sier1} and \ref{Sier2}.)
 	\begin{figure}[h]
 		\begin{center}
 			\includegraphics[width=12cm]{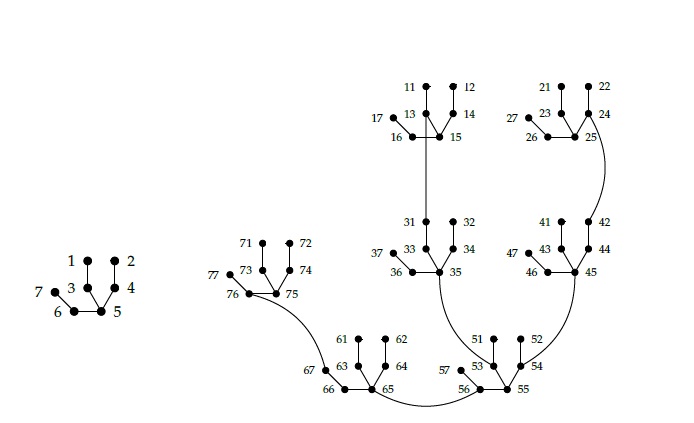}
 			\caption{{\scriptsize A graph $G$ and $S(G,2)$.}}
 			\label{Sier1}
 		\end{center}
 	\end{figure}
 	\begin{figure}[h]
 		\begin{center}
 			\includegraphics[width=20cm]{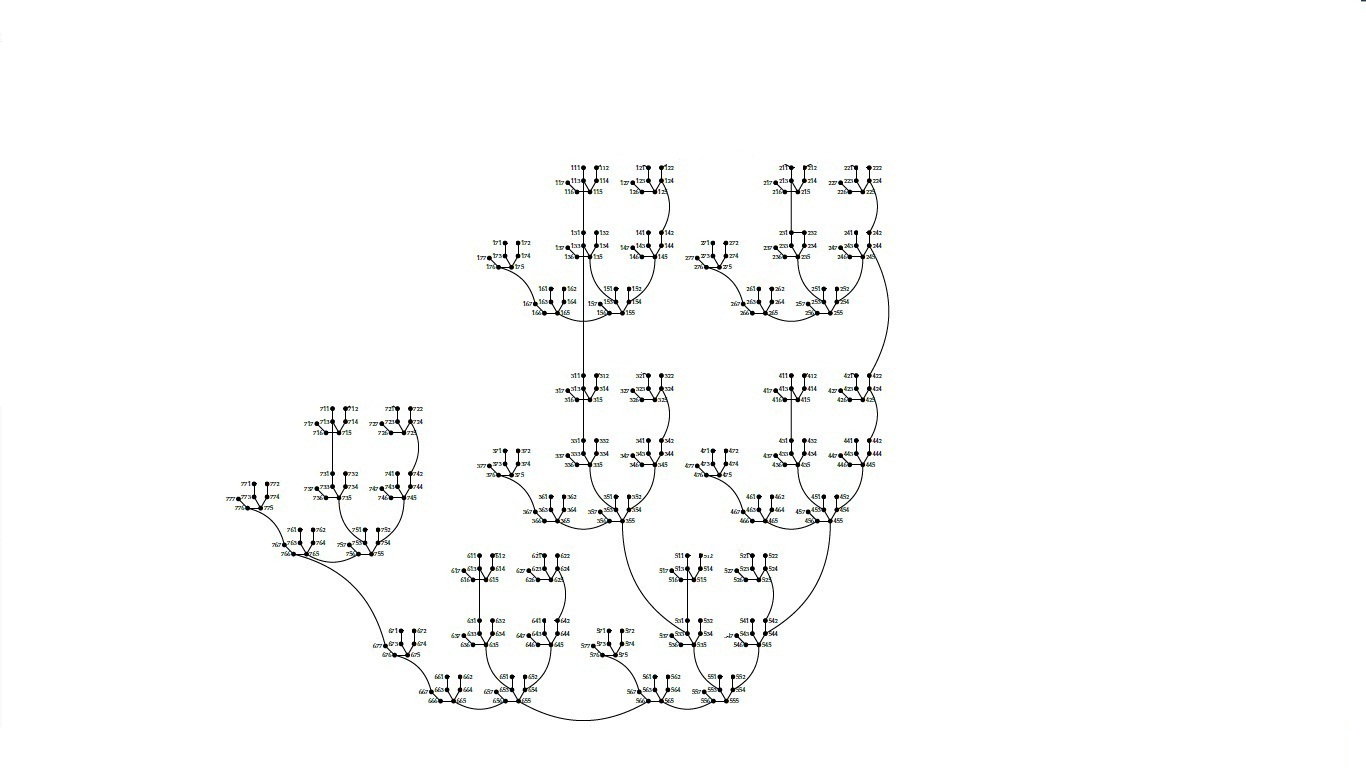}
 			\caption{{\scriptsize $S(G,3)$ for the graph $G$ in Figure \ref{Sier1}.}}
 			\label{Sier2}
 		\end{center}
 	\end{figure}	
 	 Vertices of the form $xx\ldots x$ are called extreme vertices of $S(G,t)$. Note that for any graph $G$ of order $n$ and any integer $t\geqslant 2$, $S(G,t)$ has $n$ extreme vertices and, if $x$ has degree $d(x)$ in $G$, then the extreme vertex $xx\ldots x$ of $S(G,t)$ also has degree $d(x)$.\\
 
 For any graph theoretic terminology and notations not mentioned
here, the readers may refer to \cite{Bal}. The following results are useful in this paper.\\

\begin{prop}\label{1}
	\cite{Bee} In a double Roman dominating function of weight
	$\gamma_{dR}(G)$, no vertex needs to be assigned the value $1$.
\end{prop}
\ni i.e., For any graph $G$, there exists a $\gamma_{dR}$-function with $V_{1}=\emptyset$.

\begin{thm}\label{Rdnsknt}
	\cite{Ram} For any integers $n\geqslant 2$ and $t\geqslant 1$,
	\[
	\gamma_{R}(S(K_{n}, t))\leqslant
	\begin{cases}
	\frac{2n^{t}+n-1}{n+1}, & \text{t\ \ even},\\
	\frac{2(n^{t}+1)}{n+1}, & \text{t\ \ odd}.
	\end{cases}
	\] 
\end{thm}

\section{Bounds on the Double Roman Domination Number}
First we prove a lower bound for $\gamma_{R}(S(G,t))$.
\begin{thm}\label{rdn}
	For any graph $G$ of order $n$, $\gamma_R(S(G,t))\geqslant n^{t-2} \alpha(G) \gamma_{R}(G)$, where $\alpha(G)$ is the independence number of $G$.
\end{thm}

\begin{proof}
	Let $V'\subseteq V$ be an independent set of cardinality $\alpha(G)$. For any $w\in V^{t-2}, i\in V,$ let $V_{wi}=\{wij: j\in V\}.$ Note that $\{V_{wi}: i\in V\}$ is a partition of the  vertex set of $S(G,t)$ and $<V_{wi}>\cong G,$ for every $i$ and hence there are $n^{t-1}$ disjoint copies of $G$ in $S(G, t).$ If $u$ and $v$ are adjacent in $S(G, t),$ then $u$ and $v$ are of the form $u=wxy^{r-1}, v=wyx^{r-1},$ where $w\in V^{t-r}, r\in \{1, 2,\ldots,t\}$ and $x$ and $y$ are adjacent in $G.$ Hence, for every $i,j\in V',$ none of the vertices in $V_{wi}$ is adjacent to any of the vertices in $V_{wj}.$ Also, $N(V_{wi})\cap N(V_{wj})=\emptyset,$ for $i\neq j.$  Therefore, $f(V^{t})\geqslant n^{t-2}\alpha(G)\gamma_{R}(G)$ for any RDF $f$ of $S(G,t)$  and hence $\gamma_{R}(S(G, t))\geqslant n^{t-2}\alpha(G)\gamma_{R}(G).$
\end{proof}

\begin{rem}
	It is clear that the inequality in  Theorem \ref{rdn} holds for other domination parameters like domination number, independence domination number, total domination number and many more.
\end{rem}

\begin{thm} \label{sier}
	Let $G$ be a graph of order $n$. For any $\gamma_{dR}$-function $f=(V_{0}, V_{2}, V_{3})$ on $G$, and any integer $t\geqslant 2$,$$n^{t-2}\alpha(G)\gamma_{dR}(G)\leqslant \gamma_{dR}(S(G, t))\leqslant n^{t-2}(n\gamma_{dR}(G)-|V_{3}|-|D_{3}|),$$ where $\alpha(G)$ is the independence number of $G$ and $D_{3}$ is the set of non-isolated vertices in $<V_{3}>$. 
\end{thm}

\begin{proof}
	For the left inequality, the proof is as same as that of the Theorem \ref{rdn}.	 To prove the right inequality, let $f=(V_{0}, V_{2}, V_{3})$ be a $\gamma_{dR}$-function on $G$.\\
	 
	\textbf{\emph{	Step 1:}} For a given integer $t\geqslant2$, let $S_{i}=\{wx: w\in V^{t-1}, x\in V_{i}\}, $ for $i\in \{0, 2, 3\}.$ Let $g: V^{t}\rightarrow \{0, 2, 3\}$ such that $g=(S_{0}, S_{2}, S_{3}).$ If $v\in V^{t}$ and $g(v)=0$, then $v=wy$ where $w$ is a word in $V^{t-1}$ and $y\in V_{0}$. Since $f$ is a $\gamma_{dR}$-function on $G$, there is either $z\in V_{3}\cap N_{G}(y)$ or $x_{1}, x_{2}\in V_{2}\cap N_{G}(y).$ Hence, there exists either $wz\in V_{3}\cap N_{S(G, t)}(wy)$ or $wx_{1}, wx_{2}\in V_{2}\cap N_{S(G,t)}(wy).$ Therefore, $g$ is a double Roman dominating function on $S(G,t)$ and $\gamma_{dR}(S(G, t))\leqslant g(V^{t})=n^{t-1}(2|V_{2}|+3|V_{3}|)=n^{t-1}\gamma_{dR}(G).$\\
	
	\textbf{	\emph{Step 2:}} Let $S_{3}'=\{wuu: w\in V^{t-2}, u\in V_{3}\}.$ We define $g_{1}: V^{t}\rightarrow \{0, 2, 3\}$ such that $g_{1}=(S_{0}, S_{2}\cup S_{3}', S_{3}-S_{3}').$ Let $y\in S_{0}.$ Then $y$ has the form $wuv_{0}$ where $w\in V^{t-2}, u\in V$ and $v_{0}\in V_{0}.$    Since $f$ is a $\gamma_{dR}$-function on $G,$ there is either $v_{3}\in V_{3}$ or $v_{2}, v_{2}'\in V_{2}$ in $N_{G}(v_{0}).$ Therefore, there exists either $wuv_{3}\in N_{S(G,t)}(wuv_{0})$ or $wuv_{2}, wuv_{2}'\in N_{S(G,t)}(wuv_{0}).$ If $wuv_{3}\in S_{3}-S_{3}',$ or $wuv_{2}, wuv_{2}'\in N_{S(G,t)}(wuv_{0}), $ then we are done. Now, if $wuv_{3}\in S_{3}', $ then $v_{3}=u$ and, since $v_{0}$ is adjacent to $v_{3},$ we can conclude that $y=wv_{3}v_{0}$ is adjacent to $wv_{0}v_{3}\in S_{3}-S_{3}'.$ Hence $g_{1}$ is a double Roman dominating function on $S(G, t)$. Therefore, $\gamma_{dR}(S(G, t))\leqslant g_{1}(V^{t})\leqslant n^{t-2}(n\gamma_{dR}(G)-|V_{3}|).$\\
	
	\textbf{	\emph{Step 3:}} Let $S_{3}''=\{wvv: w\in V^{t-2}, v\in D_{3}, where\  D_{3} \ is\  the\  set\  of\  non-isolated\  vertices\  in\  <V_{3}>\}.$  We define $g_{2}: V^{t}\rightarrow \{0, 2, 3\}$ such that $g_{2}=(S_{0}\cup S_{3}'', S_{2}\cup S_{3}', S_{3}-S_{3}').$
	Let $x\in V^{t}$ such that $g_{2}(x)=0.$ In this case, $g_{1}(x)=0$ or $x\in S_{3}''.$ \\
	
	\indent Suppose that $g_{1}(x)=0.$ Then $x\in S_{0}$ and hence is of the form $x=wuv_{0},$ where $w\in V^{t-2}, u\in V$ and $v_{0}\in V_{0}.$ If $N_{S(G, t)}(x)\cap S_{3}''=\emptyset,$ then there exists  $y\in N_{S(G, t)}(x)\cap (S_{3}-S_{3}')$ or $y_{1}, y_{2}\in N_{S(G, t)}(x)\cap (S_{2}\cup S_{3}').$ On the other side, if $z\in N_{S(G, t)}(x)\cap S_{3}'',$ then $z=wv_{3}v_{3},$ where $v_{3}\in D_{3},$  $u=v_{3}$ and $v_{3}$ is adjacent to $v_{0},$ which implies that $x=wv_{3}v_{0}$ is adjacent to $wv_{0}v_{3}$ and we know that $g_{2}(wv_{0}v_{3})=g_{1}(wv_{0}v_{3})=g(wv_{0}v_{3})=3.$\\
	
	\indent Now, if $x\in S_{3}'',$ then $x=wvv,$ where $w\in V^{t-2}$ and $v\in D_{3}.$ So, by definition of $D_{3},$ $x$ must be adjacent to $wvu$ for some $u\in D_{3}-\{v\}.$ Also, $g_{2}(wvu)=g_{1}(wvu)=g(wvu)=f(u)=3.$ Therefore, $g_{2}$ is a double Roman dominating function on $S(G, t),$ and so $\gamma_{dR}(S(G, t))\leqslant n^{t-2}(n\gamma_{dR}(G)-|V_{3}|-|D_{3}|).$
\end{proof}
\section{The Particular Case of Complete Graphs}
We begin this section by proving the Roman domination number of $S(K_{n}, 2)$.

\begin{thm}
	$\gamma_{R}(S(K_{n}, 2))=2n-1$.
\end{thm}

\begin{proof}
	By Theorem \ref{Rdnsknt}, we can easily deduce that $\gamma_{R}(S(K_{n}, 2))\leqslant 2n-1$.	For the reverse inequality, let $V(K_{n})=\{u_{1}, u_{2},\ldots,u_{n}\}$. Then $S(K_{n}, 2)$ is a graph with vertex set $\{u_{i}u_{j}: i,j \in \{1, 2, \ldots, n\}\}$ and edge set $\{\{u_{i}u_{j}, u_{j}u_{i}\}: u_{i}, u_{j}\in V, i\neq j\} \cup \{\{u_{i}u_{j}, u_{i}u_{k}\}: u_{i}, u_{j}, u_{k} \in V, j\neq k\}$. Note that $G_{i}=\{u_{i}u_{j}: j=1,2,\ldots,n\}$ induces a copy of $K_{n}$ for each $i=1,2,\ldots,n$ and vertex $u_{i}u_{i}$ is an extreme vertex of $S(K_{n},2)$ for each $i.$ Let $f=(V_{0},V_{1}, V_{2})$ be any $\gamma_{R}$-function of $S(K_{n},2).$ Since $d_{S(K_{n},2)}(u_{i}u_{i})=d_{K_{n}}(u_{i})=n-1$, to  Roman dominate the extreme vertex $u_{i}u_{i}$,  $G_{i}$  must contain either a vertex of value $2$ or $u_{i}u_{i}$ is  of value $1$. If every $G_{i}$ contains a vertex of value $2$,  then $f(V(S(K_{n},2)))\geqslant 2n$, which is a contradiction. Therefore, there exists at least one $G_{i}, $ say $G_{i_{0}},$ which contains exactly one vertex in $V_{1}$ (and all other vertices are in $V_{0}$).  Then, by the property of extreme vertex, the vertex in $V_{1}$ is  $u_{i_{0}}u_{i_{0}}.$ So, to  Roman dominate $u_{i_{0}}u_{j},$ for each $j\neq i_{0},$ $u_{j}u_{i_{0}}\in V_{2}.$  Therefore, $\gamma_{R}(S(K_{n}, 2))\geqslant 2(n-1)+2=2n-1.$ 
	Hence, the result.
\end{proof}

\begin{thm}
	$\gamma_{dR}(S(K_{n}, 2))=3n-1$.
\end{thm}

\begin{proof}
By Theorem \ref{sier}, we can easily deduce that $\gamma_{dR}(S(K_{n}, 2))\leqslant 3n-1$, since $|V_{3}|=1$ and $|D_{3}|=0$.	For the reverse inequality, let $V(K_{n})=\{u_{1}, u_{2},\ldots,u_{n}\}$. Then $S(K_{n}, 2)$ is a graph with vertex set $\{u_{i}u_{j}: i,j \in \{1, 2, \ldots, n\}\}$ and edge set $\{\{u_{i}u_{j}, u_{j}u_{i}\}: u_{i}, u_{j}\in V, i\neq j\} \cup \{\{u_{i}u_{j}, u_{i}u_{k}\}: u_{i}, u_{j}, u_{k} \in V, j\neq k\}$. Note that $G_{i}=\{u_{i}u_{j}: j=1,2,\ldots,n\}$ induces a copy of $K_{n}$ for each $i=1,2,\ldots,n$ and vertex $u_{i}u_{i}$ is an extreme vertex of $S(K_{n},2)$ for each $i.$ Let $f$ be any $\gamma_{dR}$-function of $S(K_{n},2).$ To double Roman dominate the extreme vertex  $u_{i}u_{i},$ $G_{i}$ must contain either one vertex having value $3,$ or two vertices having value $2$ or $u_{i}u_{i}$ is of value $2.$ If every $G_{i}$ contains a vertex of value $3$ or two vertices of value 2, then $f(V(S(K_{n},2))\geqslant 3n$ which is a contradiction to Theorem \ref{sier}. Therefore, there exists at least one $G_{i}, $ say $G_{i_{0}},$ which contains exactly one vertex having value $2.$ Then, by the property of extreme vertex the value $2$ is assigned to $u_{i_{0}}u_{i_{0}}.$ To double Roman dominate $u_{i_{0}}u_{j},$ for each $j\neq i_{0},$ $u_{j}u_{i_{0}}$ must be in $V_{2}\cup V_{3}.$ If $u_{j}u_{i_{0}}\in V_{2},$ then $G_{j}$ contains at least one more vertex of value at least $2.$ Therefore, it is optimal to assign the value $3$ to $u_{j}u_{i_{0}},$ for every $j\neq i_{0}.$ But then $f(V(S(K_{n},2))\geqslant 2+3(n-1).$
Hence, $\gamma_{dR}(S(K_{n}, 2))=3n-1.$ 
\end{proof}

\ni \textbf{Acknowledgements:} 
The first author
thanks University Grants Commission for granting fellowship under
Faculty Development Programme (F.No.FIP/$12^{th}$ plan/KLMG 045 TF
07 of UGC-SWRO). \\

\end{document}